\documentclass[12pt]{amsart}
\usepackage{graphicx}
\usepackage{amssymb}
\usepackage{epstopdf}

\usepackage{amsmath}
\usepackage{amsthm}
\usepackage{a4wide}
\usepackage{pdfsync}
\usepackage{hyperref}

\usepackage[pagewise,mathlines]{lineno}

\newcommand{\vp}{\varepsilon}

\newtheorem{lemma}{Lemma}[section]
\newtheorem{theorem}{Theorem}[section]
\newtheorem{remark}{Remark}

\def\zz{\mathbb{Z}}
\def\rr{\mathbb{R}}

\def\md{\mathcal{D}}
\def\th{\theta}

\def\vp{\varphi}

\def\la{\langle}
\def\ra{\rangle}

\def\sp{\mathcal{X}^p}
\def\hh{\mathbb{H}^n}

\def\rem{\mathcal{R}}

\title[Asymptotic expansions]{Asymptotic  expansions for  anisotropic heat kernels}
\author[L. I. Ignat and E.  Zuazua]{Liviu I. Ignat and Enrique Zuazua}

\address{L. I. Ignat
\hfill\break\indent Institute of Mathematics ``Simion Stoilow'' of the Romanian Academy\\
\hfill\break\indent  21 Calea Grivitei Street \\010702 Bucharest \\ Romania 
\hfill\break\indent \and
\hfill\break\indent 
 BCAM - Basque Center for Applied Mathematics\\
 \hfill\break\indent  Alameda de Mazarredo, 14
E-48009 Bilbao, Basque Country - Spain.}
 
 \email{{\tt
liviu.ignat@gmail.com}\hfill\break\indent  {\it Web page: }{\tt
http://www.imar.ro/\~\,lignat}}

\address{E. Zuazua
\hfill\break\indent 
 BCAM - Basque Center for Applied Mathematics\\
 \hfill\break\indent  Alameda de Mazarredo, 14
E-48009 Bilbao, Basque Country - Spain
\hfill\break\indent \and
\hfill\break\indent  Ikerbasque, Basque Foundation for Science,\\
 \hfill\break\indent Alameda Urquijo 36-5, Plaza Bizkaia, 48011, Bilbao, Basque Country, Spain}

  \email{{\tt zuazua@bcamath.org }
\hfill\break\indent {\it Web page: }{\tt http://www.bcamath.org/zuazua/}}

\begin{document}
\maketitle

\begin{abstract}
We obtain the asymptotic expansion of the solutions of some  anisotropic heat equations  
 when the initial data belong to polynomially weighted $L^p$-spaces. We mainly address two model examples.  In the first one, the diffusivity is of order two in some variables but higher in the other ones. 
In the second one we consider the heat equation on the Heisenberg group.  
\end{abstract}

\section{Introduction}

This paper is devoted to the asymptotic behavior as time tends to infinity for the Cauchy problem associated with some anisotropic heat equations, with initial data in polynomially weighted $L^p$ spaces.

There are a number of results on  asymptotic behavior of the classical isotropic heat equation on the whole space:
\begin{equation}\label{heat}
\left\{
\begin{array}{ll}
\displaystyle u_t(z,t)=\Delta u(z,t)& z\in \rr^N, t>0,\\[10pt]
u(z,0)=f(z),&,z\in \rr^N.
\end{array}
\right.
\end{equation}

The key to obtain a complete asymptotic expansion of solutions  as $t\to \infty$  is a decomposition of the initial data on the basis constituted by   Dirac's delta and its derivatives, proved in \cite{MR1183805} and that we recall for the sake of completeness. The asymptotic expansion reads as follows:

\textit{Let $G(\cdot,t)$ be the heat kernel. 
For any $1\leq p\leq N/(N-1)$ and $k\geq 0$ an integer  the solution of problem \eqref{heat} satisfies:
\begin{align*}
\Big\| u(\cdot,t)-\sum _{|\alpha|\leq k}\frac {(-1)^{|\alpha|}}{\alpha !} \Big(\int _{\rr^N} f(z)z^\alpha dz\Big) D^\alpha G(\cdot, t)\Big\| _{L^p(\rr^N)}\lesssim t^{-\frac {k+1}2}\||z|^{k+1} f\|_{L^p(\rr^N)}
\end{align*}
for any initial data $f\in L^1(\rr^N,1+|z|^k)$ such that $|z|^{k+1}f\in L^p(\rr^N)$.}

\medskip

We refer to \cite{MR1183805} for a proof and various extensions. 

According to this expansion the asymptotic behavior of solutions of the heat equation can be described in terms of the Gaussian heat kernel and its space derivatives, using the moments of the initial data as coefficients.

This decomposition is of isotropic nature; all space variables are treated equally, with the same homogeneity. 
This fact  does not permit  dealing with anisotropic heat kernels that would require different decomposition results, weighting the various euclidean variables differently, or even with isotropic heat kernels, as the one associated with the heat equation (\ref{heat}), but with initial data in anisotropically weighted spaces. 

In this paper we present new decomposition results for the initial data that lead both to new results for the heat equation (\ref{heat}) and that, simultaneously, allow handling a number of anisotropic heat kernels.

We focus mainly on two model examples to illustrate the key ideas and the phenomena that emerge due to anisotropy.

%
The first one is  the following diffusion equation, of mixed order, namely order two in some of the space variables, and order four in the other ones:
\begin{equation}\label{heat.an.intro}
\left\{
\begin{array}{ll}
\displaystyle u_t(z,t)=\Delta _x^2 u(z,t)+\Delta _y u(z,t),& z=(x,y)\in \rr^m\times \rr^n, t>0,\\[10pt]
u(z,0)=f(z),&z\in \rr^{m+n}.
\end{array}
\right.
\end{equation}
The solution $u$ can be written in convolution form  as $u(t)=G_t\ast f$ where $G_t$ is the fundamental solution
$$G_t(x,y)=t^{-\frac m2-\frac n4}G_1 (\frac x{t^{1/4}},\frac y{t^{1/2}})$$
where the similarity profile $G_1$ is a smooth function.

Based on  the decomposition obtained in \cite{MR1183805} we get a new one (see Lemma \ref{paso2} below) allowing us to prove the following asymptotic expansion.
\begin{theorem}
\label{culemea}
Let $k$ be a positive odd integer.
For any $f\in L^1(\rr^{m+n},1+|x|^{k+1}+|y|^{\frac{k+1}2})$ the solution $u$ of system \eqref{heat.an.intro}
satisfies
\begin{align}\label{est.mea}
\Big\| u(\cdot,t)- \sum _{|\beta|+2|\gamma|\leq  k}&  \frac{(-1)^{|(\beta,\gamma)|}}{\beta !\gamma !}  \Big(\int _{\rr^N} f(x,y)x^\beta y^\gamma  dxdy\Big)D_{xy}^{\beta \gamma}G_t\Big\| _{L^1(\rr^N)}\\
\nonumber &\lesssim t^{-\frac{k+1}4} \|(1+|x|^{k+1}+|y|^{\frac{k+1}2}) f\|_{L^1(\rr^N)}
\end{align}
\end{theorem}

Observe that the assumptions on the initial data weigh differently the various euclidean variables, taking into account the inhomogeneity of the diffusion operator. The resulting decomposition is in agreement with the different weight that the euclidean variables have in the fundamental solution, so that the overall contribution of its different derivatives  is balanced.
For simplicity we only have stated the result in  the $L^1(\rr^{m+n})$-norms but similar results can be stated in $L^p(\rr^{{m+n}})$ spaces under some restrictions on the exponent $p$. The decomposition in  \cite{MR1183805} 
would be insufficient to obtain such results since it is of isotropic nature, thus weighting equally all space variables, which is not sufficient to deal with the anisotropy of the heat kernel under consideration. 
The method presented in this paper allows treating separately the euclidean variables with different homogeneity.  

As we shall see, the decomposition used to obtain this result can also be applied in the case of the classical isotropic heat equation, leading to new results on its asymptotic decomposition as $t\to \infty$.

\medskip
The second problem we address is the heat equation on the Heisenberg group:
\begin{equation}\label{heat-heis}
\left\{
\begin{array}{ll}
\displaystyle u_t(z,\theta,t)=\Delta _{\mathbb{H}^n}u(z,\theta,t)& (z,\theta)\in \rr^{2n}\times \rr, t>0,\\[10pt]
u(z,\theta,0)=f(z,\theta),&, (z,\theta)\in \rr^{2n}\times \rr.
\end{array}
\right.
\end{equation}

Let us recall that, 
 as for the  classical heat equation,  the heat semi-group on $\hh$ is given by  the convolution between the initial datum and 
 the fundamental solution, a function in the Schwartz class (see \cite{MR0461589}). 

\begin{theorem}\label{representation}(\cite{MR0461589})
There exists a function $H\in \mathcal{S}(\rr^{2n+1})$  such that, $u$, the solution of the  heat equation on the Heisenberg group \eqref{heat-heis} is given by
$$u(\cdot,t)=f\ast H_t,$$
where $\ast$ is the convolution on the Heisenberg group while $H_t$ is defined by
$$H_t(z,\th)=\frac 1{t^{n+1}}H(\frac {z}{\sqrt t},\frac \th t).$$
Moreover, the function $H$ can be computed explicitly (cf. \cite{MR0461589})
$$H((z,\th))=\frac 1{(4\pi)^{n+1}}\int _\rr \Big (\frac {2\sigma}{\sinh ( {2\sigma })}\Big)^n
\exp\Big(\frac {i\sigma \th}{2}-\frac{|z|^2}{2}\frac {\sigma}{\tanh{2\sigma}}\Big)d\sigma.$$
\end{theorem}

In this case
we obtain the following asymptotic expansion. 
\begin{theorem}
\label{asimp.heis}
For any $f\in L^1(\rr^{2n+1},1+|z|^{2}+|\theta|)$ the  solution $u$ of problem \eqref{heat-heis} with initial data $f$ satisfies the
 following 
\begin{align}\label{est.mea.heis}
\Big\| u(\cdot,t)-\Big(\int _{\hh} f\Big) H_t + \sum_{j=1}^{2n} \Big (\int_{\hh}  z_j & f(z,\theta)\Big)Z_j H_t
\Big\| _{L^1(\rr^{2n+1})}\\
\nonumber &\lesssim t^{-1} \|(1+|z|^{2}+|\theta|) f\|_{L^1(\rr^{2n+1})}, \quad t>0,
\end{align}
where $Z_j$, $j=1,\dots, 2n$ are the first $2n$ vector fields of the Lie algebra on $\mathbb{H}^n$.
\end{theorem}

The results presented here can be extended to more general situations, for example to the 
 heat equation on homogeneous Carnot groups. We refer to  \cite{francesco.rossi} where this issue is addressed, although the results in that paper do not take into account the different homogeneity in what concerns the needed assumptions on the initial data. Note also the existing results for heat equations in heterogeneous media. For instance in \cite{OZ} by combining the results in  \cite{MR1183805} and Bloch-wave expansions a complete asymptotic expansion is given for the heat equation in periodic media. We  also refer to \cite{CARPIO} for similar expansions in the context of
 Vlasov -- Poisson -- Fokker -- Planck equations.

The paper is organized as follows. In Section \ref{dec} we consider the classical isotropic heat equation to illustrate that the use of the new decomposition results we present here leads to new asymptotic expansions as $t \to \infty$ that can not be derived out of  \cite{MR1183805}. In Section \ref{aniso} we consider the anisotropic model \eqref{heat.an.intro}, proving Theorem \ref{culemea} and we compare it with the asymptotic decomposition results one could get out of the isotropic decomposition results in  \cite{MR1183805}.
In the  last section we present some classical facts about the Heisenberg group and give the proof of the corresponding asymptotic expansion stated in Theorem \ref{asimp.heis}.

\section{Decomposition Lemmas}\label{dec}

In this section we obtain some variants of the results in  \cite{MR1183805} on the decomposition of functions defined in the euclidean space in the basis of Dirac's deltas and its derivatives. We shall derive some applications to the asymptotic behavior of the isotropic heat equation for initial data in anisotropically weighted spaces.

In what follows we denote by $\mathcal{D}(\rr^N)$ the space of $C^\infty$ and compactly supported functions and by $\mathcal{D}'(\rr^N)$ its dual.

We first recall the following decomposition \cite{MR1183805}.

\begin{lemma}\label{DZ} (\cite{MR1183805})
Let $k$ be a positive integer. 
For any $f\in \md(\rr^N)$ it holds
\begin{equation}\label{rep.1}
f=\sum _{|\alpha|\leq k}\frac{(-1)^{|\alpha|}}{\alpha !} (\int _{\rr^N} f(x)x^\alpha dx)D^\alpha \delta_0 +
\sum _{|\alpha|=k+1}D^\alpha F_\alpha,
\end{equation}
 where the functions $F_\alpha$, defined for $|\alpha|=k+1$, are given by
\begin{equation}\label{rep.f}
F_\alpha(x)=(k+1)\frac{(-1)^{k+1}}{\alpha !}\int _0^1 (1-t)^{k}(\frac xt)^{\alpha} f(\frac xt)\frac {dt}{t^N}.
\end{equation}
Moreover, for any  $p\in [1,\frac N{N-1})$ the following holds
$$\|F_\alpha\|_{L^p(\rr^N)}\leq \||x|^{k+1}f\|_{L^p(\rr^N)}.$$
\end{lemma}

\begin{remark}By density, the above result holds for all $f\in L^1(\rr^N,1+|x|^k)$ with $|x|^{k+1}f\in L^p(\rr^N)$.
\end{remark}

The above result is  a consequence of Taylor's formula with integral remainder, applied to any test function $\varphi$:
\begin{equation}\label{taylor}
\varphi (z)=\sum _{|\alpha|\leq k} (D^\alpha \varphi )(0)\frac {z^\alpha} {\alpha !} +(k+1)\sum _{|\alpha|=k+1}\frac {z^\alpha}{\alpha!} \int _0^1 (1-t)^k (D^\alpha \varphi)(tz)dt, z\in \rr^N.
\end{equation}

We now derive some anisotropic variants. For, given $z\in \rr^N$, $N\geq 2$ we decompose it as  $z=(x,y)\in \rr^m\times \rr^n$ with $m+n=N$. We then use a splitting argument and apply first   Taylor's expansion in the $y$ variable with $x$ fixed and then perform  Taylor's expansion in $x$. 

\begin{lemma}\label{mytaylor}
For any $\varphi\in \mathcal{D}(\rr^N)$ the following holds
\begin{align*}
\varphi(z)= &\sum _{|\alpha|\leq k} (D^\alpha \varphi)(0)\frac {z^\alpha} {\alpha !} +
\sum _{|\gamma|\leq k} \Big[(k+1-|\gamma|) \sum _{|\beta|= k+1-\gamma} \frac {x^\beta}{\beta!} \int _0^1 (1-t)^{k+|\gamma|}
(D_x^\beta D_y^\gamma \varphi)(tx,0)dt
\Big]\frac {y^\gamma}{\gamma !}\\
&\quad +(k+1)\sum _{|\gamma|=k+1} \frac {y^\gamma} {\gamma !} \int _0^1 (1-t)^k(D_y^\gamma \varphi)(x,ty)dt, \quad z=(x,y)\in \rr^N.
\end{align*}
\end{lemma}

\begin{proof}Using Taylor's expansion in $y$ with $x$ fixed we get 
\begin{align*}
\varphi(z)=&\varphi(x,y)=\sum _{|\gamma|\leq k} (D^\gamma _y \varphi)(x,0)\frac {y^\gamma}{\gamma !}+
(k+1)\sum _{|\gamma|=k+1} \frac {y^\gamma} {\gamma !} \int _0^1 (1-t)^k(D_y^\gamma \varphi)(x,ty)dt.
\end{align*}

Now, applying Taylor's expansion with respect to $x$ we get:
\begin{align*}
\varphi(z)
=&\sum _{|\gamma|\leq k}  \Big[\sum _{|\beta|\leq k-|\gamma|} (D_x^\beta D_y^\gamma \varphi)(0,0)\frac {x^\beta}{\beta !}\frac {y^\gamma}{\gamma !}
\Big]\\
&+
\sum _{|\gamma|\leq k} \Big[(k+1-|\gamma|) \sum _{|\beta|= k+1-\gamma} \frac {x^\beta}{\beta!} \int _0^1 (1-t)^{k+|\gamma|}
(D_x^\beta D_y^\gamma \varphi)(tx,0)dt
\Big]\frac {y^\gamma}{\gamma !}\\
&\quad +(k+1)\sum _{|\gamma|=k+1} \frac {y^\gamma} {\gamma !} \int _0^1 (1-t)^k(D_y^\gamma \varphi)(x,ty)dt.
\end{align*}

Writing $\alpha=(\beta,\gamma)$ we obtain the desired identity.
\end{proof}

We now introduce the space $ L^p(\rr^m,|x|^{|\beta|};\ L^1(\rr^n,|y|^{|\gamma|}))$, constituted by  the functions $f$ such that $|x|^{|\beta|} 
\||y|^{|\gamma|} f(x,\cdot)\|_{L^1(\rr^n)}$ belongs to $L^p(\rr^m)$, i.e.:
\begin{equation}\label{norm.mixed}
\|f\|_{L^p(\rr^m,|x|^{|\beta|};\ L^1(\rr^n,|y|^{|\gamma|}))}^p=\int _{\rr^m}|x|^{p|\beta|}\|f(x,\cdot)\|_{L^1(\rr^n; |y|^{|\gamma|} )}^p
dx<\infty.
\end{equation}

Following the same ideas of treating differently the different variables $x$ and $y$ we   get the following result.

\begin{lemma}\label{my.descoposition}
For any $f\in \md(\rr^{m+n})$ it holds
\begin{align}\label{rep.2}
f(x,y)=& \sum _{|(\beta, \gamma)|\leq k}\frac{(-1)^{|(\beta,\gamma)|}}{\beta !\gamma !}  \Big(\int _{\rr^N} f(x,y)x^\beta y^\gamma  dxdy\Big)D_{xy}^{\beta \gamma} \delta_0\\
\nonumber &+\sum _{|\gamma|\leq k}\sum _{|\beta|= k+1-|\gamma|} \frac{(-1)^{|\gamma|}}{\gamma !}D^\beta_x [\rem f]_{\beta \gamma}(x)D_y^\gamma \delta_0(y)\\
\nonumber &+\sum _{|\gamma|=k+1}D^\gamma_y F_\gamma(x,y), \qquad  (x,y)\in \rr^m\times \rr^n,
\end{align}
where  the remainder terms are defined as follows:
$$[\rem f]_{\beta \gamma}(x)=\frac{(-1)^{|\beta|}|\beta|}{\beta !}\int _0^1 (1-t)^{|\beta|-1}(\frac xt)^{\beta} \int _{\rr^n}f(\frac xt,y)y^\gamma dy \frac {dt}{t^m}$$
 for $|\gamma|\leq k, |\beta|+|\gamma|=k+1$, 
and 
$$F_\gamma(x,y)=(k+1)\frac{(-1)^{k+1}}{\gamma !}\int _0^1 (1-t)^{k}(\frac yt)^{\gamma} f(x,\frac yt)\frac {dt}{t^n}$$
for $|\gamma|=k+1$.

Moreover, for $1\leq p< \frac n{n-1}$ it holds
\begin{equation}\label{ord.k+1}
\|F_\gamma \|_{L^p(\rr^N)}\leq \||y |^{k+1}f \|_{L^p(\rr^N)}
\end{equation}
and for $1\leq p< \frac m{m-1}$ it holds
\begin{align}\label{mab}
\|[\rem f]_{\beta\gamma} \|_{L^p(\rr^m)}\leq \|f\|_{ L^p(\rr^m,|x|^{|\beta|};\ L^1(\rr^n,|y|^{|\gamma|}))}.
\end{align}

\end{lemma}

\begin{remark}
The main term is the same as in Lemma \ref{DZ}, since
$$\sum _{|(\beta, \gamma)|\leq k}\frac{(-1)^{|(\beta,\gamma)|}}{\beta !\gamma !}  \Big(\int _{\rr^N} f(x,y)x^\beta y^\gamma  dxdy\Big)D_{xy}^{\beta \gamma} \delta_0=\sum _{|\alpha|\leq k}\frac{(-1)^{|\alpha|}}{\alpha !}\Big (\int _{\rr^N} f(z)z^\alpha dz\Big)D^\alpha \delta_0.
$$
But estimating the remainder terms $F_\gamma$ and $[Mf]_{\beta\gamma}$  different anisotropic spaces from those of Lemma \ref{DZ} are needed. 

By density, the decomposition holds for all $f\in L^1(\rr^N,1+|\cdot|^{k})$ that belongs to 
  $$\sp=\bigcap _{|\gamma|\leq k}   L^p(\rr^m,|x|^{k+1-|\gamma|};\ L^1(\rr^n,|y|^{|\gamma|})) \bigcap
  L^p(\rr^m;\ L^p(\rr^n,|y|^{k+1})) .$$
  We also define the natural norm on $\sp$  $$\| f \|_{\sp}:=\sum _{0\leq |\gamma|\leq k}
 \|f\|_{ L_x^p(\rr^m,|x|^{k+1-|\gamma|};\ L^1_y(\rr^n,|y|^{|\gamma|}))}+\||y|^{k+1}f\|_{L^p(\rr^N)}.$$
  When $p=1$ the space $\sp$ is exactly the one needed in Lemma \ref{DZ}, i.e. $L^1(\rr^{m+n}, |(x,y)|^{k+1})$.
\end{remark}


\begin{proof}[Proof of Lemma \ref{my.descoposition}] We write $\rr^N$ as $\rr^m\times \rr^n$.
We first fix $x\in \rr^m$ and apply Lemma \ref{DZ} to the function $y\rightarrow f(x,y)$. Then
\begin{equation}\label{dec.2.2}
f(x,y)=\sum _{|\gamma|\leq k}\frac{(-1)^{|\gamma|}}{\gamma !} \Big(\int _{\rr^n} f(x,y)y^\gamma dy\Big)D_y^\gamma \delta_0(y) +
\sum _{|\gamma|=k+1}D^\gamma_y F_\gamma(x,y),
\end{equation}
where
\begin{equation}\label{rep.f.gamma}
F_\gamma(x,y)=(k+1)\frac{(-1)^{k+1}}{\gamma !}\int _0^1 (1-t)^{k}(\frac yt)^{\gamma} f(x,\frac yt)\frac {dt}{t^n}.
\end{equation}
A similar argument as in \cite{MR1183805} shows that  for any  $x\in \rr^m$ and $p\in [1,\frac {n}{n-1})$ we have 
\begin{equation}\label{Fab}
\|F_\gamma(x,\cdot)\|_{L^p(\rr^n)}\leq \||\cdot|^{k+1}f(x,\cdot)\|_{L^p(\rr^n)}.
\end{equation}
Taking the $L^p(\rr^m)$-norm in the $x$-variable  we obtain estimate \eqref{ord.k+1}.

We denote by $[Mf]_{\gamma}(x)$ the $y$-moment of order $\gamma$ of $f(x,\cdot)$ in $\rr^n$:
\begin{equation}
\label{ymom} 
  [Mf]_{\gamma}(x)=\int _{\rr^n} f(x,y)y^\gamma dy.
\end{equation}
 We apply Lemma \ref{DZ} to the function $x\rightarrow [Mf]_{\gamma}(x)$ and obtain
$$[Mf]_{\gamma}(x)=\sum _{|\beta|\leq k-|\gamma|}\frac{(-1)^{|\beta|}}{\beta !} \Big(\int _{\rr^m} [Mf]_{\gamma}(x)x^\beta dx\Big)D_x^\beta \delta_0(x) +
\sum _{|\beta|=k+1-|\gamma|}D^\beta_x [\rem f]_{\beta \gamma}(x)
$$
where $[\rem f]_{\beta \gamma}$ is given by
$$[\rem f]_{\beta \gamma}(x)=\frac{(-1)^{|\beta|}|\beta|}{\beta !}\int _0^1 (1-t)^{|\beta|-1}(\frac xt)^{\beta} [Mf]_{\gamma}(\frac xt)\frac {dt}{t^m}.$$
Writing explicitly the function $[Mf]_{\gamma}$, we have 
\begin{equation}\label{starstar}
[Mf]_{\gamma}(x)=\sum _{|\beta|\leq k-|\gamma|}\frac{(-1)^{|\beta|}}{\beta !} \Big(\int _{\rr^N} f(x,y)x^\beta y^\gamma  dxdy\Big)D_x^\beta \delta_0(x) +
\sum _{|\beta|=k+1-|\gamma|}D^\beta_x [\rem f]_{\beta \gamma}(x)
\end{equation}
where $[\rem f]_{\beta\gamma}$ are given by
$$[\rem f]_{\beta \gamma}(x)=\frac{(-1)^{|\beta|}|\beta|}{\beta !}\int _0^1 (1-t)^{|\beta|-1}(\frac xt)^{\beta} \int _{\rr^n}f(\frac xt,y)y^\gamma dy \frac {dt}{t^m}.$$
Replacing $[Mf]_{\gamma}$ in \eqref{dec.2.2} with identity \eqref{starstar} we obtain 
\begin{align*}
f
=&\sum _{|\gamma|\leq k}\sum _{|\beta|\leq k-|\gamma|} \frac{(-1)^{|\gamma|+|\beta|}}{\gamma !\beta!} \Big(\int _{\rr^N} f(x,y)x^\beta y^\gamma  dxdy\Big)D_x^\beta \delta_0(x)D_y^\gamma \delta_0(y)\\
&+\sum _{|\gamma|\leq k}\sum _{|\beta|= k+1-|\gamma|} \frac{(-1)^{|\gamma|}}{\gamma !}D^\beta_x [\rem f]_{\beta \gamma}(x)D_y^\gamma \delta_0(y)+\sum _{|\gamma|=k+1}D^\gamma_y F_\gamma(x,y)\\
=& \sum _{|(\beta, \gamma)|\leq k}\frac{(-1)^{|(\beta,\gamma)|}}{\beta !\gamma !}  \Big(\int _{\rr^N} f(x,y)x^\beta y^\gamma  dxdy\Big)D_{xy}^{\beta \gamma} \delta_0\\
&+\sum _{|\gamma|\leq k}\sum _{|\beta|= k+1-|\gamma|} \frac{(-1)^{|\gamma|}}{\gamma !}D^\beta_x [\rem f]_{\beta \gamma}(x)D_y^\gamma \delta_0(y)+\sum _{|\gamma|=k+1}D^\gamma_y F_\gamma(x,y).
\end{align*}
Proceeding as in \cite{MR1183805} and using Minkowski's inequality, for any $1\leq p<\frac m{m-1}$ we get estimate
 \eqref{mab} on the remainder $[\rem f]_{\beta\gamma}$ 
 \begin{align*}
\|[\rem f]_{\beta\gamma} \|_{L^p(\rr^m)}&\leq \Big\||x|^\beta \int _{\rr^n} f(x,y) y^\gamma dy|\Big\|_{L^p(\rr^m)}
\leq \|f\|_{ L^p(\rr^m,|x|^{|\beta|};\ L^1(\rr^n,|y|^{|\gamma|}))}.
\end{align*}
The proof is now finished.
\end{proof}

\medskip
We now apply the new decomposition obtained in Lemma \ref{my.descoposition} 
 to the heat equation \eqref{heat}.

\begin{theorem}\label{new.app.to.heat}Assume   that   $N=m+n$.
For any $1\leq p<\min \{\frac m{m-1},\frac n{n-1}\}$ and $k\geq 0$ integer, the solution of problem \eqref{heat} satisfies:
\begin{align*}
\Big\| u(\cdot,t)-\sum _{|\alpha|\leq k}\frac {(-1)^{|\alpha|}}{\alpha !} \Big(\int _{\rr^N} f(z)z^\alpha dz\Big) D^\alpha G(\cdot, t)\Big\| _{L^p(\rr^N)}\lesssim t^{-\frac {k+1}2}\| f\|_{\sp}
\end{align*}
for any initial data $f\in L^1(\rr^N,1+|z|^k)\cap\sp$. 
\end{theorem}

\begin{remark}
In the case $p=1$ the above estimate is the same as the one obtained in \cite{MR1183805} by applying Lemma \ref{DZ}. 

The range  of exponents  $p$ for which the above Theorem applies  is larger than that in \cite{MR1183805} but the space $\sp$ where the initial data is taken is more complex and of anisotropic nature.
\end{remark}
\begin{proof}
We now apply Lemma \ref{my.descoposition}  to the initial data $f$. Then
\begin{align*}
\Big\|G(t)\ast & f-  \sum _{|\alpha|\leq k}\frac{(-1)^{|\alpha|}}{\alpha !} (\int _{\rr^N} f(x)x^\alpha dx)D^\alpha G(t) \Big\|_{L^p(\rr^N)}\\
&\leq \sum _{0\leq |\gamma|\leq k}
\sum _{|\beta|= k+1-|\gamma|} \Big\| G(t)\ast  \Big( D^\beta_x [\rem f]_{\beta \gamma}(x)D_y^\gamma \delta_0(y)\Big)\Big\|_{L^p(\rr^N)}+\sum _{|\gamma|=k+1}\|G(t)\ast D^\gamma_y F_\gamma(x,y)\|_{L^p(\rr^N)}\\
&\leq \sum _{0\leq |\gamma|\leq k}
\sum _{|\beta|= k+1-|\gamma|} \| D^\beta_x D^\gamma _yG(t,\cdot,y)\ast_x   [ \rem f]_{\beta \gamma}(\cdot)\|_{L^p(\rr^N)}+\sum _{|\gamma|=k+1}\|D^\gamma_yG(t)\ast  F_\gamma(x,y)\|_{L^p(\rr^N)}\\
&\leq  \sum _{0\leq |\gamma|\leq k}
\sum _{|\beta|= k+1-|\gamma|}\Big\| \|D_x^\beta D_y^\gamma G(t) \|_{L^1_x} \|[\rem f]_{\beta \gamma}\|_{L^p_x}\Big \|_{L^p_y}+|t|^{-\frac{k+1}2}\|F_\gamma\|_{L^p(\rr^N)}\\
&\leq  \sum _{0\leq |\gamma|\leq k}
\sum _{|\beta|= k+1-|\gamma|} \|[ \rem f]_{\beta \gamma}\|_{L^p_x}\|D_x^\beta D_y^\gamma G(t)\|_{L^p_yL^1_x} +|t|^{-\frac{k+1}2}\||y|^{k+1}f\|_{L^p(\rr^N)}\\
&\leq  \sum _{0\leq |\gamma|\leq k}
\sum _{|\beta|= k+1-|\gamma|}|t|^{-\frac {|\beta|+|\gamma|}2-\frac n2(1-\frac 1p) }\|f\|_{ L^p(\rr^m,|x|^{|\beta|};\ L^1(\rr^n,|y|^{|\gamma|}))}+|t|^{-\frac{k+1}2}\||y|^{k+1}f\|_{L^p(\rr^N)}\\
&\leq t^{-\frac{k+1}2}\|f\|_{\sp}.
\end{align*}
The proof of Theorem \ref{new.app.to.heat} is now finished.
\end{proof}

\section{An anisotropic heat equation of mixed order}\label{aniso}


This section is devoted to the analysis of the following  system
\begin{equation}\label{heat.an}
\left\{
\begin{array}{ll}
\displaystyle u_t(z,t)=\Delta _x^2 u(z,t)+\Delta _y u(z,t),& z=(x,y)\in \rr^m\times \rr^n, t>0,\\[10pt]
u(z,0)=f(z),&z\in \rr^{m+n}.
\end{array}
\right.
\end{equation}
The solution $u$ of the above system is given by $u(t)=G_t\ast f$ where $G_t$ is the fundamental solution given by
$$G_t(x,y)=t^{-\frac m4-\frac n2}G_1 (\frac x{t^{1/4}},\frac y{t^{1/2}}),$$
and the self-similar profile  $G_1$ is so that its Fourier transform is such that
$$\hat {G_1}(\xi,\eta)=e^{-|\xi|^4-|\eta|^2}.$$

We now apply Lemma \ref{DZ} to the initial data $f$. For any  $1\leq p<N/(N-1)$, $N=m+n$, we obtain that
$u$, solution of \eqref{heat.an}, satisfies
\begin{align*}
\Big\| u(\cdot,t)-\sum _{|\alpha|\leq k}\frac {(-1)^{|\alpha|}}{\alpha !} \Big(\int _{\rr^N} f(z)z^\alpha dz\Big) D^\alpha G_t(\cdot, t)\Big\| _{L^p(\rr^N)}\lesssim ||z|^{k+1} f\|_{L^p(\rr^N)}\sum _{|\alpha|=k+1} \|D^\alpha G_t\|_{L^1(\rr^N)},
\end{align*}
or, writing each index $\alpha$ of length $N=n+m$ as $\alpha=(\beta,\gamma)\in \zz^m\times \zz^n$,
\begin{align*}
\Big\| u(\cdot,t)- \sum _{|(\beta, \gamma)|\leq k}\frac{(-1)^{|(\beta,\gamma)|}}{\beta !\gamma !}  \Big(\int _{\rr^N} f(x,y)x^\beta y^\gamma  dxdy\Big)D_{xy}^{\beta \gamma}G_t\Big\| _{L^p(\rr^N)}\\
\\
\lesssim ||z|^{k+1} f\|_{L^p(\rr^N)}\sum _{|\beta|+|\gamma|=k+1} \|D^\beta_x D^\gamma_y G_t\|_{L^1(\rr^N)}.
\end{align*}
Remark that, due to the anisotropy of the fundamental solution, the derivatives in the $x$-variable decay differently
 from those in the $y$-variable:
\begin{equation}\label{decay.xy}
\|D^\beta_x D^\gamma_y G_t \|_{L^q(\rr^N)}\simeq t^{-\frac{|\beta|} 4-\frac {|\gamma|}2-\frac N2(1-\frac 1q)}, \quad q\in [1,\infty].
\end{equation}
As a consequence, for large $t$ we have
\begin{align*}
\Big\| u(\cdot,t)- &\sum _{|(\beta, \gamma)|\leq k}\frac{(-1)^{|(\beta,\gamma)|}}{\beta !\gamma !}  \Big(\int _{\rr^N} f(x,y)x^\beta y^\gamma  dxdy\Big)D_{xy}^{\beta \gamma}G_t\Big\| _{L^p(\rr^N)}\\
&
\lesssim ||z|^{k+1} f\|_{L^p(\rr^N)}\sum _{|\beta|+|\gamma|=k+1}  t^{-\frac{|\beta|} 4-\frac {|\gamma|}2}\lesssim
t^{-\frac{k+1}4} \||z|^{k+1} f\|_{L^p(\rr^N)}.
\end{align*}
Note however  that in the left hand side of the above inequality there are terms that decay faster than the remainder $t^{-\frac{k+1}4} $. Accordingly it is natural to  split the terms in the left hand side in two parts: those that decay faster and, respectively,  slower than
  $t^{-\frac{k+1}4} $.  
  
  Define
  $$\Lambda(p,k)=\{(a,b)\in \zz^2: a,b\geq 0,  k+1-2N(1-1/p)\leq a+2b,\, a+b\leq k \}.$$
It follows that, as long as $f$ belongs to some weighted spaces, the solution $u$ of system \eqref{heat.an} satisfies for large time $t$ the following estimate
\begin{align}\label{est.1}
\Big\| u(\cdot,t)- &\sum _{|\beta|+2|\gamma|< k+1-2N(1-1/p)}\frac{(-1)^{|(\beta,\gamma)|}}{\beta !\gamma !}  \Big(\int _{\rr^N} f(x,y)x^\beta y^\gamma  dxdy\Big)D_{xy}^{\beta \gamma}G_t\Big\| _{L^p(\rr^N)}\\
\nonumber&
\lesssim 
t^{-\frac{k+1}4} \||z|^{k+1} f\|_{L^p(\rr^N)}
+\sum _{ (\beta,\gamma)\in \Lambda (p,k) } t^{-\frac{|\beta|}4-\frac{|\gamma|}2-\frac N2(1-\frac 1p)}
\|f\|_{L^1(\rr^N,|x|^{|\beta|}|y|^{|\gamma|})}\\
\nonumber &\lesssim t^{-\frac{k+1}4} \Big( \||z|^{k+1} f\|_{L^p(\rr^N)}+\sum _{ (\beta,\gamma)\in \Lambda (p,k) }
\|f\|_{L^1(\rr^N,|x|^{|\beta|}|y|^{|\gamma|})}\Big).
\end{align}
We now state rigorously the result to which this analysis leads. For the sake of simplicity we  state it for $p=1$ only.

\begin{theorem}
\label{culemaDZ}
For any $f\in L^1(\rr^{n+m},1+|x|^{k+1}+|y|^{k+1})$ the following holds for  large time $t$:
\begin{align*}
\Big\| u(\cdot,t)- &\sum _{|\beta|+2|\gamma|\leq  k}\frac{(-1)^{|(\beta,\gamma)|}}{\beta !\gamma !}  \Big(\int _{\rr^N} f(x,y)x^\beta y^\gamma  dxdy\Big)D_{xy}^{\beta \gamma}G_t\Big\| _{L^1(\rr^N)}\\
&\lesssim t^{-\frac{k+1}4} \|(1+|x|^{k+1}+|y|^{k+1}) f\|_{L^1(\rr^N)}
\end{align*}
\end{theorem}

\begin{proof}
We apply estimate \eqref{est.1} to obtain the following:
\begin{align*}
\Big\| u(\cdot,t)- &\sum _{|\beta|+2|\gamma|\leq k}\frac{(-1)^{|(\beta,\gamma)|}}{\beta !\gamma !}  \Big(\int _{\rr^N} f(x,y)x^\beta y^\gamma  dxdy\Big)D_{xy}^{\beta \gamma}G_t\Big\| _{L^1(\rr^N)}\\
&\lesssim t^{-\frac{k+1}4} \Big( \||z|^{k+1} f\|_{L^1(\rr^N)}+\sum _{ (\beta,\gamma)\in \Lambda (p,k) }
\|f\|_{L^1(\rr^N,|x|^{|\beta|}|y|^{|\gamma|})}\Big)\\
&\lesssim t^{-\frac{k+1}4} \Big(\|(|x|^{k+1}+|y|^{k+1}) f\|_{L^1(\rr^N)}+\|(|x|^{k}+|y|^{k}) f\|_{L^1(\rr^N)}
\Big).
\end{align*}
Since $f$ belongs to $L^1(\rr^N)$ we obtain the desired result.
\end{proof}

\medskip

Observe that, in view of \eqref{decay.xy}, we need twice more derivatives in $x$ than in $y$ to obtain the same decay of the kernel $G_t$. We now give a new decomposition that, when applied to the anisotropic equation \eqref{heat.an}, distinguishes the two variables $x$ and $y$.  

We consider the case when $k$ is an odd integer, $k+1=2m+2$. Since we have seen that 
$\|D^\alpha_x G_t\|_{L^p(\rr^N)}\simeq \|D^\beta _y G_t\|_{L^p(\rr^N)} , |\alpha|=2|\beta|=2m,$   we will take into account that, for the decay rate of the kernel, one derivative in $y$ has the same effect as two derivatives in $x$.

\begin{lemma}\label{paso2}
Let $k$  be a  positive odd integer. For any $f\in \md(\rr^{m+n})$ it holds
\begin{align}\label{rep.3}
f(x,y)=& \sum _{|\beta|+2 |\gamma|\leq k}\frac{(-1)^{|(\beta,\gamma)|}}{\beta !\gamma !}  \Big(\int _{\rr^N} f(x,y)x^\beta y^\gamma  dxdy\Big)D_{xy}^{\beta \gamma} \delta_0\\
\nonumber &+\sum _{|\gamma|\leq (k-1)/2}\sum _{|\beta|+2|\gamma|=k+1} \frac{(-1)^{|\gamma|}}{\gamma !}D^\beta_x [\rem f]_{\beta \gamma}(x)D_y^\gamma \delta_0(y)+\sum _{|\gamma|=(k+1)/2}D^\gamma_y F_\gamma(x,y)
\end{align}
where  the remainder terms are defined as follows: 
$$F_\gamma(x,y)=\frac {k+1}2\frac{(-1)^{\frac{k+1}2}}{\gamma !}\int _0^1 (1-t)^{\frac{k-1}2}(\frac yt)^{\gamma} f(x,\frac yt)\frac {dt}{t^n},
$$
for $|\gamma|=\frac{k+1}2$, 
and   
$$[\rem f]_{\beta \gamma}(x)=\frac{(-1)^{|\beta|}|\beta|}{\beta !}\int _0^1 (1-t)^{|\beta|-1}(\frac xt)^{\beta} \int _{\rr^n}f(\frac xt,y)y^\gamma dy \frac {dt}{t^m}$$
defined for $\ |\gamma|\leq (k-1)/2, |\beta|+2|\gamma|=k+1$.

Moreover, for $1\leq p< \frac n{n-1}$, it holds
\begin{equation}\label{paso2.ord.k+1}
\|F_\gamma \|_{L^p(\rr^N)}\leq \||y |^{\frac{k+1}2}f(x,\cdot )\|_{L^p(\rr^N)},
\end{equation}
and, for $1\leq p< \frac m{m-1}$, 
\begin{align}\label{paso2.mab}
\|[\rem f]_{\beta\gamma} \|_{L^p(\rr^m)}\leq \|f\|_{ L^p(\rr^m,|x|^{|\beta|};\ L^1(\rr^n,|y|^{|\gamma|}))}.
\end{align}
\end{lemma}

\begin{remark}
By density the results hold true for a larger class of functions $f$.
\end{remark}

\begin{proof}[Proof of Lemma \ref{paso2}] Set $k+1=2l+2$.
Fix $x\in \rr^m$ and apply Lemma \ref{DZ} to the function $y\rightarrow f(x,y)$. Then
\begin{equation}\label{ff}
f(x,y)=\sum _{|\gamma|\leq l}\frac{(-1)^{|\gamma|}}{\gamma !}\Big (\int _{\rr^n} f(x,y)y^\gamma dy\Big)D_y^\gamma \delta_0(y) +
\sum _{|\gamma|=l+1}D^\gamma_y F_\gamma(x,y),
\end{equation}
where
\begin{equation}\label{2.rep.f.gamma}
F_\gamma(x,y)=(l+1)\frac{(-1)^{l+1}}{\gamma !}\int _0^1 (1-t)^{l}(\frac yt)^{\gamma} f(x,\frac yt)\frac {dt}{t^n}, \ |\gamma|=l+1.
\end{equation}
Observe that for any  $x\in \rr^m$ and $p\in [1,\frac {n}{n-1})$ we have 
\begin{equation}\label{2.Fab}
\|F_\gamma(x,\cdot)\|_{L^p(\rr^n)}\leq \||y|^{l+1}f(x,\cdot)\|_{L^p(\rr^n)}.
\end{equation}
Taking the $L^p(\rr^m)$-norm in the $x$ variable we obtain estimate \eqref{paso2.ord.k+1}.

Recall the definition of $[Mf]_{\gamma}$ given in \eqref{ymom}. We apply Lemma \ref{DZ} to the function $x\rightarrow [Mf]_{\gamma}(x)$, with a remainder of order $k+1-|\gamma|$  and obtain that
\begin{equation}\label{mm}
[Mf]_{\gamma}(x)=\sum _{|\beta|\leq k-2|\gamma|}\frac{(-1)^{|\beta|}}{\beta !} \Big(\int _{\rr^m} [Mf]_{\gamma}(x)x^\beta dx\Big)D_x^\beta \delta_0(x) +
\sum _{|\beta|=k+1-2|\gamma|}D^\beta_x [\rem f]_{\beta \gamma}(x),
\end{equation}
where
$$[\rem f]_{\beta \gamma}(x)=\frac{(-1)^{|\beta|}|\beta|}{\beta !}\int _0^1 (1-t)^{|\beta|-1}(\frac xt)^{\beta} [Mf]_{\gamma}(\frac xt)\frac {dt}{t^m}.$$

Using decomposition \eqref{mm} in \eqref{ff} we obtain that $f$ can be written as  
\begin{align*}
f(x,y)
=&\sum _{|\gamma|\leq m}\frac{(-1)^{|\gamma|}}{\gamma !} \Big(\sum _{|\beta|\leq k-2|\gamma|}\frac{(-1)^{|\beta|}}{\beta !} (\int _{\rr^N} f(x,y)x^\beta y^\gamma  dxdy)D_x^\beta \delta_0(x) +\\
&
\quad \quad \quad \quad \quad \quad \quad \quad \quad+\sum _{|\beta|=k+1-2|\gamma|}D^\beta_x [\rem f]_{\beta \gamma}(x)\Big)D_y^\gamma \delta_0(y) \\
&+
\sum _{|\gamma|=m+1}D^\gamma_y F_\gamma(x,y)\\
=&\sum _{|\gamma|\leq m}\sum _{|\beta|\leq k-2|\gamma|} \frac{(-1)^{|\gamma|}}{\gamma !} \frac{(-1)^{|\beta|}}{\beta !} \Big(\int _{\rr^N} f(x,y) x^\beta y^\gamma  dxdy\Big)D_x^\beta \delta_0(x)D_y^\gamma \delta_0(y)\\
&+\sum _{|\gamma|\leq m}\sum _{|\beta|= k+1-2|\gamma|} \frac{(-1)^{|\gamma|}}{\gamma !}D^\beta_x [\rem f]_{\beta \gamma}(x)D_y^\gamma \delta_0(y)+\sum _{|\gamma|=m+1}D^\gamma_y F_\gamma(x,y)\\
=& \sum _{|\beta|+2|\gamma|\leq 2m+1}\frac{(-1)^{|(\beta,\gamma)|}}{\beta !\gamma !}  \Big(\int _{\rr^N} f(x,y)x^\beta y^\gamma  dxdy\Big)D_{xy}^{\beta \gamma} \delta_0\\
&+\sum _{|\gamma|\leq m}\sum _{|\beta|= 2m+2-|\gamma|} \frac{(-1)^{|\gamma|}}{\gamma !}D^\beta_x [\rem f]_{\beta \gamma}(x)D_y^\gamma \delta_0(y)+\sum _{|\gamma|=m+1}D^\gamma_y F_\gamma(x,y).
\end{align*}
Estimate \eqref{paso2.mab} on the remainder $[\rem f]_{\beta\gamma}$ is obtained as follows: for any $1\leq p< \frac m{m-1}$ we have
\begin{align*}
\|[\rem f]_{\beta\gamma} \|_{L^p(\rr^n)}&\leq \Big\||x|^\beta \int _{\rr^n} f(x,y) y^\gamma dy\Big\|_{L^p(\rr^n)}
\leq \|f\|_{ L^p(\rr^n,|x|^{|\beta|};\ L^1(\rr^m,|y|^{|\gamma|}))}.
\end{align*}
The proof is now complete.
\end{proof}

We now apply this new decomposition to obtain an asymptotic expansion for 
the solutions of \eqref{heat.an}, and prove Theorem \ref{culemea}.

\begin{proof}[Proof of Theorem \ref{culemea}]
We now apply Lemma \ref{paso2}  to our initial data $f\in L^1(\rr^N,1+|x|^{k+1}+|y|^{\frac{k+1}2})$. Then
\begin{align*}
\Big\|G_t\ast & f-  \sum _{|\beta|+2|\gamma|\leq k}\frac{(-1)^{|(\beta,\gamma)}}{\beta !\gamma!}\Big (\int _{\rr^N} f(x)x^\beta  y^\gamma dxdy\Big)D_{xy}^{\beta \gamma}G_t \Big\|_{L^1(\rr^N)}\\
&\leq \sum _{ |\gamma|\leq (k-1)/2}
\sum _{|\beta|+2|\gamma|=k+1}\Big \| G_t\ast  \Big( D^\beta_x [\rem f]_{\beta \gamma}(x)D_y^\gamma \delta_0(y)\Big)\Big\|_{L^1(\rr^N)}\\
&\quad \quad +\sum _{|\gamma|=(k+1)/2}\|G_t\ast D^\gamma_y F_\gamma(x,y)\|_{L^1(\rr^N)}\\
&\leq\sum _{ |\gamma|\leq (k-1)/2}
\sum _{|\beta|+2|\gamma|=k+1} \| D^\beta_x D^\gamma _yG_t (\cdot,y)\ast_x   [\rem f]_{\beta \gamma}(\cdot)\|_{L^1(\rr^N)}\\
&\quad \quad +\sum _{|\gamma|=(k+1)/2}\|D^\gamma_yG_t\ast  F_\gamma(x,y)\|_{L^1(\rr^N)}\\
&\leq  \sum _{0\leq |\gamma|\leq k}
\sum _{|\beta|= k+1-|\gamma|}\Big\| \|D_x^\beta D_y^\gamma G_t \|_{L^1_x(\rr^m)} \|[\rem f]_{\beta \gamma}\|_{L^1_x(\rr^m)}\Big \|_{L^1_y(\rr^n)}+|t|^{-\frac{k+1}4}\sum _{|\gamma|=(k+1)/2} \|F_\gamma\|_{L^1(\rr^N)}\\
&\leq \sum _{ |\gamma|\leq (k-1)/2}
\sum _{|\beta|+2|\gamma|=k+1} \|[ \rem f]_{\beta \gamma}\|_{L^1_x}\|D_x^\beta D_y^\gamma G_t\|_{L^1(\rr^N)} +|t|^{-\frac{k+1}4}
\||y|^{\frac{k+1}2}f\|_{L^1(\rr^N)}\\
&\leq  \sum _{ |\gamma|\leq (k-1)/2}
\sum _{|\beta|+2|\gamma|=k+1} |t|^{-\frac {|\beta|+2|\gamma|}4 }\|f\|_{ L^1(\rr^N,|x|^{|\beta|}|y|^{|\gamma|})}+|t|^{-\frac{k+1}4}\||y|^{\frac{k+1}2}f\|_{L^1(\rr^N)}\\
&\leq t^{-\frac{k+1}4} \|(1+|x|^{k+1}+|y|^{\frac{k+1}2}) f\|_{L^1(\rr^N)}.
\end{align*}
The proof of Theorem \ref{culemea} is now complete.
\end{proof}

\begin{remark}
Remark that  all the terms in the left hand side of the  main estimate of Theorem \ref{culemea}
have a slower decay than those of the right hand side. Also observe that, with respect to Theorem \ref{culemaDZ}, in Theorem \ref{culemea} we have assumed only integrability of  $(k+1)/2$ moments in the variable $y$.
\end{remark}

\section{The Heisenberg Group}\label{heis}

First of all we give the definition and describe the main aspects of  the Heisenberg group. Given $(z,\th)$ of $\rr^{2n+1}$ in the form
$(z,\theta)=(z_1,\dots,z_n,z_{n+1},\dots,z_{2n},\th)$, the composition law on the group is given by 
$$(z,\th)\circ (z',\th')=(z+z',\th+\th'+2\sum_{j=1}^n (z_{n+j}z_j' -z_jz'_{n+j})).$$
In this way $(\rr^{2n+1},\circ)$ is a Lie group, whose identity element is the origin $(z,\theta)= (0, 0)$, the inverse being given by 
$(z,\th)^{-1}=(-z,-\th)$.

Let us now consider the dilations 
$$\delta_\lambda :\rr^{2n+1}\rightarrow \rr^{2n+1}, \ \delta_\lambda(z,\th)=(\lambda z,\lambda^2 \th).$$
We have that $\delta_\lambda$ is an automorphism on $(\rr^{2n+1},\circ)$ for every $\lambda>0$. In this way $\mathbb{H}^n=(\rr^{2n+1},\circ,\delta_\lambda)$ is a homogeneous space.
The Lie algebra on $\mathbb{H}^n$ is given by the left invariant vector fields
$$Z_j=\partial_{z_j}+2z_{n+j}\partial _\th,\ Z_{n+j}=\partial _{z_{n+j}}-2z_j\partial _\th, \ j=1,\dots, n, \Theta=\partial _\th.$$

The convolution product of two functions $f$ and $g$ on $\hh$ is defined by
$$(f\ast g)(w)=\int _{\hh} f(w\circ v^{-1} )g(v)dv=\int _{\hh} f(v)g(v^{-1}\circ w)dw,$$
where $v^{-1}$ is the inverse of element $v$ with respect to the group operation $\circ$ on $\hh$.
Here, $dw$ is the Haar measure on $\hh$ which is exactly the euclidian measure on $\rr^{2n+1}$.
It should be emphasized that the convolution on the Heisenberg group is not commutative. However if $P$ is a left invariant vector field on $\hh$ then one sees easily that
$$P(f_1\ast f_2)=f_1\ast Pf_2=Pf_1\ast f_2.$$

The canonical sub-laplacian on $\mathbb{H}^n$ is given by 
$$\Delta_{\mathbb{H}^n}=\sum _{j=1}^{2n}Z_j^2.$$
In the case of the Heisenberg group $\mathbb{H}^n$ the exponential and logarithmic maps (see \cite{MR2363343}, Ch. 1, p.~48, for the definition of the ${\rm Exp}$ and ${\rm Log}$ maps on  general homogeneous Lie groups) are given by  (see \cite{MR2363343}, p.~167 and Remark 3.2.4 on p.~163)
\begin{equation}\label{exp}
{\rm Exp}\Big (\sum _{j=1}^{2n}  z_j Z_j+\th \Theta \Big)=(z,\theta), z=(z_1,\dots,z_{2n})
\end{equation}
and
\begin{equation}\label{log}
{\rm Log}(z,\th)=\sum _{j=1}^{2n}  z_j Z_j+\th \Theta.
\end{equation}


We point out that properties \eqref{exp} and \eqref{log} are not true in general for homogeneous Lie groups as we can see in the following example.
Following \cite{MR2363343}, Section 3.2, p.~158 let us set $\rr^N=\rr^m\times \rr^n$  and denote its points $z=(x,\theta)$ with $x\in \rr^m$ and $\theta\in \rr^n$. Given a $m\times m$ matrix $B$ with real entries let 
$$(x,\theta )\circ (\xi, \tau)=(x+\xi,\theta+\tau+\frac 12 \la Bx, \xi \ra),$$
where $\la \cdot ,\cdot  \ra$ stands for the inner product in $\rr^m$. Again $\delta_\lambda :\rr^{N}\rightarrow \rr^{N}, \ \delta_\lambda(x,\theta)=(\lambda x,\lambda^2 \theta)$ is an automorphism of $(\rr^N,\circ)$ for any $\lambda >0$. Then $(\rr^N,\circ, \delta_\lambda)$
is a homogeneous Lie group of step two. In this case the $\rm{Exp}$ and $\rm {Log}$ functions are given by (see (3.11), p.~167 in \cite{MR2363343})
\[ {\rm Exp} ((\xi,\tau)\cdot Z) = (\xi, \tau +\frac 14 \la B\xi ,\xi\ra)
\]
and
\[ {\rm Log} (x,\theta) = (x,\theta -\frac 14 \la B\xi ,\xi\ra)\cdot Z
\]
where
\[(\xi,\tau)\cdot Z= \sum _{j=1}^m \xi _jX_j +\sum _{i=1}^n \tau _i\Theta_i
\]
and $X_1,\dots,X_m, \Theta_1,\dots, \Theta_n$ is the Jacobian basis. 

We can see that \eqref{exp} and \eqref{log} hold if and only if $B$ is a skew-adjoint matrix. In the case of $\mathbb{H}^n$ this holds since the matrix $B$ is given by
\[B=4
\left(\begin{array}{cc}0 & \mathbb{I}_n \\-\mathbb{I}_n & 0\end{array}\right).
\]
In the particular case of skew-symmetric matrices $B$ we believe  that the results of this section are still valid.

\medskip

We now recall some results concerning Taylor's formula with integral remainder on homogenous Carnot groups $(\mathbb{G},\circ)$.

\begin{lemma}[Lemma 20.3.8, \cite{MR2363343}]\label{taylor.1}
Let $x,g\in \mathbb{G}$ and $u\in C^{m}(\mathbb G,\rr)$. Then we have
\begin{equation}\label{tay.1}
u(x\circ h )=\sum _{k=0}^{m}\frac {1}{k!}(({\rm Log}\ h)^k u)(x)+
\frac 1 {m!}\int _0^1 (1-s)^m (({\rm Log}\ h)^{m+1} u)(x\circ {\rm Exp} (s {\rm Log}\ h))ds.
\end{equation}
\end{lemma}
Let us now apply this Lemma to the case of the Heisenberg group $\hh$.
Using the structure on $\mathbb{H}^n$ we obtain that for any $h\in \mathbb{H}^n$ and $s\in \rr$ the following holds 
$${\rm Exp} (s {\rm Log}\ h)=sh.$$
Thus Lemma \ref{taylor.1}  (see Corollary 20.3. 9 in \cite{MR2363343}, p.~755) gives us the following:
\begin{lemma}\label{taylor.2}
Let $x,h\in\mathbb{H}^n$ and $u\in C^{m+1}(\mathbb{H}^n,\rr)$. Then we have
\begin{align*}
u(x\circ h )=&\sum _{k=0}^{m} \sum _{i_1,\dots, i_k=1}^{2n+1} \frac {h_{i_1}\dots h_{i_k}} {k!} (Z_{i_1}\dots Z_{i_k}u)(x)\\
&+\sum _{i_1,\dots, i_{m+1}=1}^{2n+1} \frac {h_{i_1}\dots h_{i_{m+1}}} {(m+1)!}
\int _0^1 (1-s)^m(Z_{i_1}\dots Z_{i_{m+1}}u)( x\circ sh)ds,
\end{align*}
where $h_{2n+1}=\theta$  and $Z_{2n+1}=\Theta$.
\end{lemma}

We now follow the same ideas as in the case of the anisotropic heat equation making partial decompositions in some of the involved variables. Here, in the context of the Heisenbergh group, variables $z=(z_1,\dots, z_{2n})$ and $z_{2n+1}=\theta$ will be treated separately. For simplicity   we treat here only the case $m=1$.

\begin{lemma}\label{dec.heis}
For any $\vp\in \mathcal{D}(\rr^{2n+1})$ the following identity holds:
\begin{align*}
\vp(z,\th)=&\vp(0,0)+\sum _{j=1}^{2n}z_j (Z_j \vp)(0,0)\\
&+\int _0^1 (1-s)\sum _{j,k=1}^{2n} z_jz_k (Z_jZ_k\vp)(sz,0)ds+\int _0^1 \th (\Theta \vp )((z,s\th))ds.
\end{align*}
\end{lemma}

\begin{remark}
We emphasize the difference with formula (20.63)  in \cite{MR2363343}, p.~761:
\begin{align*}
\vp(z,\th)=&\vp(0,0)+\sum _{j=1}^{2n}z_j (Z_j \vp)(0,0)+\th (\Theta \vp) (0,0)+\sum _{j,k=1}^{2n} \int _0^1 (1-s)z_jz_k (Z_jZ_k\vp)(sz,s\th)ds.
\end{align*}
\end{remark}

\begin{proof}
We apply Lemma \ref{taylor.2}  to $x=(z,0)$ and $h=(0,\th)$, elements from $\rr^{2n}\times \rr$.  Thus
\[
\vp(z,\th)=\vp (z,0)+\int _0^1 (\th \Theta \vp )((z,0)\circ (0,s\th))ds=\vp (z,0)+\int _0^1 \th (\Theta\vp )((z,s\th))ds.
\]
We now apply Lemma  \ref{taylor.2}  with $x=(0,0)$ and $h=(z,0)$:
\begin{align*}
\vp(z,0)=\vp(0,0)+\sum _{j=1}^{2n}z_j (Z_j \vp)(0,0)+\int _0^1 (1-s)\sum _{j,k=1}^{2n} z_jz_k (Z_jZ_k\vp)(sz,0)ds.
\end{align*}
Putting  toghether the above identities we obtain that
\begin{align*}
\vp(z,t)=&\vp(0,0)+\sum _{j=1}^{2n}z_j (Z_j \vp)(0,0)+\int _0^1 (1-s)\sum _{j,k=1}^{2n} z_jz_k (Z_jZ_k\vp)(sz,0)ds\\
&+\int _0^1 \th (\Theta\vp )((z,s\th))ds
\end{align*}
and the proof is now finished.
\end{proof}

We now apply the Taylor like identity in Lemma \ref{dec.heis} to obtain a decomposition of a function $f$ into Dirac's delta basis adapted to the Lie algebra on $\hh$.
\begin{theorem}\label{dec.nou.H}
For any $f\in L^1(\rr^{2n+1},1+|z|^2+|\th|)$ the following decomposition holds:
\begin{align}\label{dec.H}
f=\Big(\int _{\rr^{2n+1}} f\Big) \delta _0 + &\sum_{j=1}^{2n} \Big (\int_{\rr^{2n+1}}  z_j f(z,\th)dzd\th \Big)Z_j \delta _0+ (\Theta F) + \sum _{j,k=1}^{2n} Z_jZ_k \big(F_{jk}\delta _0(\th)\big),
\end{align}
 where
 \begin{equation}\label{F}
F(z,\th)=-\int _0^1 \frac \th s f(z,\frac \th s)ds
\end{equation}
and
\begin{equation}\label{Fjk}
F_{jk}(z)=\int _{\rr^{}} \int _0^1  (1-s) \frac {z_j}s\frac{z_k}s f(\frac zs,\frac \th s)\frac{ds d\th}{s^{2n}}.
\end{equation}
Moreover,
$$\|F\|_{L^1(\rr^{2n+1})}\leq \|\th f\|_{L^1(\rr^{2n+1})}$$
and
$$\|F_{jk}\|_{L^1(\rr^{2n+1})}\leq \|z_jz_k f\|_{L^1(\rr^{2n+1})}.$$
\end{theorem}

\begin{proof}
First of all it is sufficient to prove the above theorem for smooth $f$ and then extend the result for any $f\in L^1(\rr^{2n+1},|z|^2+|\th|)$.
Let us choose $f$ and $\vp$ two smooth functions.
Applying Taylor's expansion obtained in Lemma \ref{dec.heis} we get the following:
\begin{align*}
\la  f,\vp \ra _{\mathcal{D}',\mathcal{D}}=&\int _{\rr^{2n+1}}f(z,\th)\Big(\vp(0,0)+\sum _{j=1}^{2n}z_j (Z_j \vp)(0,0)\Big)dzd\th\\
&+\int _{\rr^{2n+1}}f(z,\th)\int _0^1 (1-s)\sum _{j,k=1}^{2n} z_jz_k (Z_jZ_k\vp)(sz,0)dsdzd\th\\
&+\int _{\rr^{2n+1}}f(z,\th) \int _0^1 \th (\Theta \vp )(z,s\th)dsdzd\th\\
=& <\delta _{(0,0)},\vp> \int _{\rr^{2n+1}}f(z,\th)dzd\th - \sum _{j=1}^{2n} <Z_j\delta_{(0,0)},\vp> \int _{\rr^{2n+1}}f(z,\th)z_jdzd\th \\
&+\sum _{j,k=1}^{2n}\int _{\rr^{2n+1}} (Z_jZ_k\vp)(z,0)
\int _0^1 (1-s)f(\frac zs ,\th)\frac { z_j}s \frac {z_k}s \frac {ds}{s^{2n}}dzd\th\\
&+\int _{\rr^{2n+1}} (\Theta\vp )(z,\th) \int _0^1  \frac \th s f(z,\frac \th s) \frac {ds}s dzd\th.
\end{align*}
Observe now that
\begin{align*}
\int _{\rr^{2n+1}} &(Z_jZ_k\vp)(z,0)
\int _0^1 (1-s)f(\frac zs ,\th)\frac { z_j}s \frac {z_k}s \frac {ds}{s^{2n}}dzd\th\\
=&
\int _{\rr^{2n}} (Z_jZ_k\vp)(z,0)\Big[\int _{\rr}
\int _0^1 (1-s)f(\frac zs ,\th)\frac { z_j}s \frac {z_k}s \frac {dsd\th}{s^{2n}}\Big]dz\\
=&  \int _{\rr^{2n}} (Z_jZ_k\vp)(z,0) F_{jk}(z)dz= \la F_{jk}\delta _0(\th),Z_jZ_k \varphi \ra _{\mathcal{D}',\mathcal{D}}\\
=& \la Z_jZ_k \big( F_{jk}\delta_0(\th)\big),\varphi  \ra _{\mathcal{D}',\mathcal{D}}
\end{align*}
where 
$$F_{jk}(z)=\int _{\rr}
\int _0^1 (1-s)f(\frac zs ,\th)\frac { z_j}s \frac {z_k}s \frac {dsd\th}{s^{2n}}.$$
Denoting 
\[F(z,\th)=-\int _{\rr^{2n+1}} \frac \th s f(z,\frac \th s)ds
\]
we obtain the desired decomposition. The estimates of the $L^1$-norms of $F$ and  $F_{jk}$ immediately follow. 
\end{proof}

\medskip
We  now  apply the decomposition obtained in Theorem \ref{dec.nou.H} 
 to obtain the asymptotic behaviour of the solutions of the heat equation on the Heisenbergh group $\mathbb{H}^n$. 

%
%

\begin{proof}[Proof of Theorem \ref{asimp.heis}]
Using decomposition \eqref{dec.H} we have that
\begin{align*}
\Big\|u(z,\th,t)-&\Big(\int _{\rr^{2n+1}} f\Big) H_t + \sum_{j=1}^{2n} \Big (\int_{\rr^{2n+1}}  z_j  f(z,\th)\Big)Z_j H_t\Big\|_{L^1(\rr^{2n+1})}\\
&\leq \|\Theta F\ast H_t\|_{L^1(\rr^{2n+1})}+\sum _{j,k=1}^{2n}\|Z_jZ_k \big( F_{jk} \delta_0(\th)\big) \ast H_t \|_{L^1(\rr^{2n+1})}\\
&=\|F\ast \Theta H_t\|_{L^1(\rr^{2n+1})}+\sum _{j,k=1}^{2n}\| (F_{jk} \delta _0(\th))\ast (Z_jZ_k H_t) \|_{L^1(\rr^{2n+1})}\\
&\leq \|F\| _{L^1(\rr^{2n+1})}\| \Theta H_t\|_{L^1(\rr^{2n+1})}+\sum _{j,k=1}^{2n}\| F_{jk} \|_{L^1(\rr^{2n})} \| Z_jZ_k H_t\|_{L^1(\rr^{2n+1})}\\
&\leq \|\th f\|_{L^1(\rr^{2n+1})}+\sum _{j,k=1}^{2n} \|z_jz_k f\|_{L^1(\rr^{2n+1})}\leq \|(|\th|+|z|^2)f\|_{L^1(\rr^{2n+1})}.
\end{align*}
The proof of the decay rate property is now finished. 
\end{proof}

 {\bf
Acknowledgements.}

The first author was partially supported by Grant PN-II-ID-PCE-2011-3-0075 of the Romanian National Authority for Scientific Research, CNCS -- UEFISCDI, MTM2011-29306-C02-00, MICINN, Spain and ERC Advanced Grant FP7-246775 NUMERIWAVES.  

The second author
was partially supported by Grant MTM2011-29306-C02-00, MICINN, Spain, ERC Advanced Grant FP7-246775 NUMERIWAVES, ESF Research Networking Programme OPTPDE and Grant PI2010-04 of the Basque Government.

%
%
%

\end{document}